\documentclass[12pt,reqno]{amsart}
\usepackage{etex}
\usepackage{amssymb,tikz-cd}
\usepackage{mathrsfs}
\usepackage{palatino}
\usepackage{bbm}
\usepackage{pdflscape}
\usepackage{mathrsfs}
\usepackage{cases}
\usepackage{epic}
\usepackage{amsfonts}
\usepackage{graphicx}
\usepackage{amsmath}
\usepackage{amssymb, upgreek}
\usepackage{bm}
\usepackage{latexsym,todonotes}
\usepackage{pdflscape}
\usepackage[all]{xypic}
\usepackage{color}
\usepackage{colordvi}
\usepackage{multicol}
\usepackage[normalem]{ulem}
\usepackage[linktocpage=true]{hyperref}
\usepackage{hyperref}
\hypersetup{colorlinks,linkcolor=blue,urlcolor=cyan,citecolor=red}
\usepackage{pictex}
\usepackage{amscd}
\usepackage{latexsym}
\usepackage{amssymb}
\usepackage{eucal}
\usepackage{eufrak}
\usepackage{verbatim}

\newcommand{\NN}{\mathbb{N}}

\newcommand{\ZZ}{\mathbb{Z}}

\newcommand{\cF}{\mathcal{F}}

\newcommand{\fg}{\mathfrak{g}}
\newcommand{\fgl}{\mathfrak{gl}}

\newcommand{\kk}{\mathbbm{k}}

\newcommand{\fr}{\mathfrak{r}}
\newcommand{\gl}{\mathfrak{gl}}

\newcommand{\fz}{\mathfrak{z}}

\DeclareMathOperator{\super}{super}
\DeclareMathOperator{\Char}{char}

\DeclareMathOperator{\odd}{odd}

\DeclareMathOperator{\gr}{gr}

\DeclareMathOperator{\Mat}{Mat}

\DeclareMathOperator{\HC}{HC}

\numberwithin{equation}{section}
\newtheorem{Theorem}{Theorem}[section]
\newtheorem{Lemma}[Theorem]{Lemma}
\newtheorem{Corollary}[Theorem]{Corollary}
\newtheorem{Proposition}[Theorem]{Proposition}
\newtheorem{Remark}[Theorem]{Remark}

\theoremstyle{Theorem}

\newtheorem*{thm*}{Theorem}
\newtheorem*{thm**}{Corollary}
\newtheorem*{thm***}{Theorem B}
\theoremstyle{remark}

\newtheorem*{Definition}{Definition}

\numberwithin{equation}{section}

\begin{document}
\title[even char]{Super Yangians in characteristic $2$}
\author[Hao Chang \lowercase{and} Hongmei Hu]{Hao Chang \lowercase{and} Hongmei Hu*}
\address[Hao Chang]{School of Mathematics and Statistics, Central China Normal University, Wuhan 430079, China}
\email{chang@ccnu.edu.cn}
\address[Hongmei Hu]{Department of Mathematics, Shanghai Maritime University, Shanghai 201306, China}
\email{hmhu@shmtu.edu.cn}
\date{\today}
\thanks{*~Corresponding author.}
\subjclass[2020]{Primary 17B37, Secondary 17B50}

\begin{abstract}
We define the super Yangian $Y_{m|n}$ over a field $\kk$ of characteristic $2$,
and show that the super Yangian $Y_{m|n}$ is a 
deformation of the super universal enveloping algebra of the current Lie algebra $\fgl_{m+n}[t]$.
By employing the methods of the work of \cite{BT18}, 
we also give a description of the center of $Y_{m|n}$.
\end{abstract}
\maketitle
\section{Introduction}
The Yangian $Y_n$ associated to the Lie
algebra $\fgl_n$ over complex field was introduced in \cite{TF79}.
It is an associative algebra generated by the {\it RTT generators} 
$\{t_{i,j}^{(r)};~1\leq i,j\leq n,r>0\}$ subject to the relations:
\begin{align}\label{RTT relations}
\left[t_{i,j}^{(r)}, t_{k,l}^{(s)}\right] =\sum_{t=0}^{\min(r,s)-1}
\left(t_{k, j}^{(t)} t_{i,l}^{(r+s-1-t)}-
t_{k,j}^{(r+s-1-t)}t_{i,l}^{(t)}\right)
\end{align}
for $1\leq i,j,k,l\leq n$ and $r,s>0$.
In \cite{BT18},
Brundan and Topley developed the theory of the Yangian $Y_n$ over a field of positive
characteristic.
They defined the Yangian $Y_n$ over any field $\kk$ by the same {\it RTT presentation} \eqref{RTT relations} as above.
When $\kk$ is of characteristic zero,
it is well-known that $Y_n$ is a filtered deformation of the universal enveloping algebra 
$U(\fgl_n[t])$ of the {\it current algebra} $\fgl_n[t]:=\fgl_N\otimes \kk[t]$, cf. \cite{MNO96}.
The same assertion is true when $\kk$ is of positive characteristic (\cite[Lemma 4.2]{BT18}).
Moreover,
they extended the {\it Drinfeld-type presentation} (\cite{D88, BK05}) from characteristic zero to positive characteristic and gave a description of the center $Z(Y_n)$ of $Y_n$.
One of the key features which differs from characteristic zero is the existence of a large central subalgebra $Z_p(Y_n)$,
called the {\it $p$-center}.

The super Yangian $Y_{m|n}$ associated to the general linear Lie superalgebra $\fgl_{m|n}$ over the complex field was defined by Nazarov \cite{Na91} in terms of the RTT presentation as a super analogue of $Y_n$. 
In \cite{CH23},
we obtained the superalgebra generalization of the $p$-center of \cite{BT18} for the modular super Yangian of type $A$.
Following \cite{Na91}, 
we define the super Yangian over a field $\kk$ to be the associative superalgebra with the RTT generators $\{t_{i,j}^{(r)};~1\leq i,j\leq m+n, r>0\}$
subject to the following relations:
\begin{align}\label{RTT relations for Ym|n}
\left[t_{i,j}^{(r)}, t_{k,l}^{(s)}\right]=(-1)^{\bar{i}\bar{j}+\bar i\bar k+\bar{j}\bar{k}}\sum_{t=0}^{\min(r,s)-1}
\left(t_{k, j}^{(t)} t_{i,l}^{(r+s-1-t)}-
t_{k,j}^{(r+s-1-t)}t_{i,l}^{(t)}\right),
\end{align}
where $\bar i=0$ if $1\leq m$, 
$\bar i=1$ if $i\geq m+1$,
and the bracket is understood as a supercommutator.
Over a field $\kk$ of characteristic $\neq 2$,
the $Y_{m|n}$ is also a deformation of the universal enveloping algebra of the polynomial current superalgebra $\fgl_{m|n}[t]$ (cf. \cite[Corollary 2]{Gow07}, \cite[Lemma 3.2]{CH23}).

The definition of {\it Lie superalgebra} is the same as usual for any characteristic $\neq 2$,
that is, a Lie superalgebra is a $\ZZ_2$-graded algebra $L=L_{\bar 0}\oplus L_{\bar 1}$ with bilinear multiplication $[\cdot,\cdot]$ satisfying the usual skew-supersymmetry and super Jacobi identity (see for instance \cite{Kac77}).
The definition does not naively extend to characteristic $2$.
Otherwise, the Lie superalgebra is just a $\ZZ_2$-graded Lie algebra.
To address this issue, 
a Lie superalgebra over a field of characteristic $2$ is defined as a $\ZZ_2$-graded Lie algebra $L=L_{\bar 0}\oplus L_{\bar 1}$ endowed with a quadratic map $Q:L_{\bar 1}\rightarrow L_{\bar 0}$ (see \cite{Bou+20, Bou+23}).
One can define the {\it super enveloping algebra} $U_{\super}(L)$ as a replacement for the usual eveloping algebra of a Lie superalgerba.
Then we have the PBW isomorphism $\gr U_{\super}(L)\cong S(L_{\bar 0})\otimes \Lambda(L_{\bar 1})$,
which is in a good agreement with the theory of Lie superalgebras in characteristic $\neq 2$ (see \cite{EH25}).

Of course, 
the above issue will also happen for the modular super Yangian $Y_{m|n}$.
Suppose that $\Char\kk=2$.
Observing \eqref{RTT relations for Ym|n},
we see that the super phenomenon disappears and the super Yangian $Y_{m|n}$ coincides with the Yangian $Y_{m+n}$ (see \cite{CH23}, \cite{Hu25}).
In this short note, we define the super Yangian $Y_{m|n}$ in characteristic $2$.
One can show that it is the deformation of the super enveloping algebra
$U_{\super}(\fgl_{m+n}[t])$.
In Section \ref{Section:Lie super}, 
we recall some notations about current algebras
and Lie superalgebras in characteristic $2$.
Then we determine the center of the super universal enveloping algebra of the current Lie superalgebra.
In Section \ref{section:super Yangian},
we first define the super Yangian $Y_{m|n}$ in characteristic $2$ and prove the PBW theorem.
Moreover, we give a description of the center of $Y_{m|n}$.
\section{Lie superalgebras}\label{Section:Lie super}
From now on,
we always assume that $\kk$ is an algebraically closed field of characteristic $\Char(\kk)=:p=2$.
\subsection{The current Lie algebra}
Given two positive integers $m,n>0$,
we consider the general linear Lie algebra $\gl_{m+n}$.
The {\it current algebra} is defined to be the Lie algebra $\gl_{m+n}[t]:=\gl_{m+n}\otimes \kk[t]$.
We will always denote this Lie algebra by $\fg$ and write $U(\fg)$ for its enveloping algebra.
Then the elements $e_{i,j}t^r:=e_{i,j}\otimes t^r$ with $r\geq 0$ and $i,j=1,\dots,m+n$
make a basis of $\fg$ and the Lie bracket is given by
\begin{align}\label{lie bracket of g}
[e_{i,j}t^r,e_{k,l}t^s]=\delta_{k,j}e_{i,l}t^{r+s}-\delta_{l,i}e_{k,j}t^{r+s},
\end{align}
where $e_{i,j}$ is the standard elementary matrix and
$1\leq i,j,k,l\leq m+n$,
$r,s\geq 0$.
Now let the indices $i,j$ run through $1,\dots,m+n$.
Set $\bar i=0$ if $1\leq i\leq m$ and $\bar i=1$ if $m<i\leq m+n$. 
Accordingly, equip the algebra $\fg$ with the $\ZZ_2$-gradation,
the {\it parity} of the element $e_{i,j}t^r$ is found by $\bar i+\bar j\mod 2$.
So that the basis element $e_{i,j}t^r$ is {\it even} if $1\leq i,j\leq m$ or $m+1\leq i,j\leq m+n$ and $e_{i,j}t^r$ is {\it odd} otherwise.

Recall that a {\it restricted Lie algebra} over $\kk$ is a Lie algebra $\mathfrak{r}$
over $\kk$ with a map $[p]:\fr\rightarrow \fr$ sending $x\mapsto x^{[2]}$ such that the map
\[
\xi:\fr\rightarrow U(\fr);~\quad x\mapsto \xi(x):=x^p-x^{[p]}
\]
satisfies two properties:
(i) $\xi(\fr)$ lie in the center of $U(\fr)$;
(ii) $\xi$ is $p$-semilinear.
Here the $p$-semilinear means that $\xi(\lambda x)=\lambda^p\xi(x)$ and
$\xi(x+y)=\xi(x)+\xi(y)$, for all $x,y\in\fr, \lambda\in\kk$.
We refer the reader to \cite[Section 2]{Jan98} for more details on restricted Lie algebras.

We identify $\fg$ with the matrix algebra $\Mat_{m+n}(\kk[t])$.
Then the current Lie algebra $\fg$ carries a natural restricted Lie algebra structure
with the $[p]$-map defined by the rule $a^{[p]}=a^{[2]}=a\cdot a$ for $a\in \fg$.
Here $a\cdot a$ denotes the usual matrix multiplication.
In particular, the $[p]$-map defined on the basis by $(e_{i,j}t^r)^{[p]}=\delta_{i,j}e_{i,j}t^{pr}$.

We will let $Z(\fg)$ be the center of $U(\fg)$.
Using the restricted structure,
we can define the $p$-center $Z_p(\fg)$ of $U(\fg)$ to be the subalgebra of
$Z(\fg)$ generated by $x^p-x^{[p]}$ for all $x\in\fg$. 
Since the map $\xi:x\mapsto x^p-x^{[p]}$ is $p$-semilinear,
we have that
\begin{align}\label{zpg}
 Z_p(\fg):=\kk[(e_{i,j}t^r)^p-\delta_{i,j}e_{i,j}t^{rp};~1\leq i,j\leq m+n, r\geq 0]   
\end{align}
as a free polynomial algebra.
\subsection{Lie superalgebras in characteristic $2$}
We first recall the following definition.
\begin{Definition}(\cite[Subsection 1.2.3]{Bou+20}, \cite[Subsection 2.2]{Bou+23},
see also \cite[Definition 3.1]{EH25})
A Lie superalgebra over $\kk$ is a $\ZZ_2$-graded Lie algebra $L=L_{\bar 0}\oplus L_{\bar 1}$
together with a quadratic map $Q:L_{\bar 1}\rightarrow L_{\bar 0}$
such that for $y,y_1,y_2\in L_{\bar 1}$, $x\in L$ we have
\begin{align}\label{quadratic q}
[y_1,y_2]=Q(y_1+y_2)-Q(y_1)-Q(y_2),\quad [Q(y),x]=[y,[y,x]].  
\end{align}
The {\it super universal enveloping algebra} of $L$ is defined via 
\begin{align}\label{def:super enveloping algebra}
U_{\super}(L):=U(L)/(y^2-Q(y);~y\in L_{\bar 1}).  
\end{align}
\end{Definition}
As before, we identify the elements in the current Lie algebra $\fg$ with polynomial matrices.
We therefore obtain
\[
 (x+y)^{[p]}=(x+y)^{[2]}=(x+y)\cdot (x+y)=x^{[2]}+y^{[2]}+x\cdot y+y\cdot x=x^{[2]}+y^{[2]}+[x,y]
\]
for all $x,y\in\fg$.
The last equality follows from the definition of Lie bracket of $\fg$ \eqref{lie bracket of g}.
On the other hand, since $\fg$ is restricted,
it follows that $[x,[x,y]]=[x^{[2]},y]$.
Consequently, the restriction $[p]_{\fg_{\bar 1}}$ of the $[p]$-map to $\fg_{\bar 1}$
defines a quadratic map satisfying \eqref{quadratic q}
and the Lie algebra $\fg$ is endowed with a Lie superalgebra structure.
By definition \eqref{def:super enveloping algebra}, we have
\[
U_{\super}(\fg)=U(\fg)/(y^2-y^{[2]};~y\in\fg_{\bar 1}).
\]
Let $Z_{p,\odd}(\fg)$ be the subalgebra of $Z_p(\fg)$ generated by $\xi(y)=y^2-y^{[2]}$ for all $y\in\fg_{\bar 1}$.
Again,
the semilinearity of $\xi$ implies that
\[
 Z_{p,\odd}(\fg)=\kk[(e_{i,j}t^r)^2;~\bar i+\bar j=1\!\!\!\!\! \mod 2, r\geq 0]      
\]
is a free polynomial algebra. 

For homogeneous elements $x_1,\dots,x_s$ in a superalgebra $A$,
a {\em supermonomial} in $x_1,\dots,x_s$ means
a monomial of the form $x_1^{i_1}\dots x_s^{i_s}$ for some $ i_1,\dots,i_s\in\ZZ_{>0}$ and $i_j\leq 1$ if $x_j$ is odd.
The following proposition is
therefore a consequence of the PBW theorem (cf. \cite[Propositions 2.3 and 2.8]{Jan98}).
\begin{Proposition}\label{prop:PBW for u superg}
The enveloping algebra $U(\fg)$ is free as a module over $Z_{p,\odd}(\fg)$ with basis given by the ordered supermonomials in the following elements of $\fg$
   \[
\{e_{i,j}t^r;~1\leq i,j\leq m+n, r\geq 0\}.
   \]
In particular, the above ordered supermonomials forms a linear basis for $U_{\super}(\fg)$.
\end{Proposition}
\subsection{The center of $U_{\super}(\fg)$}
The adjoint action of $\fg$ on itself extends uniquely to an action of $\fg$ on $U(\fg)$.
Since the ideal $(y^2-y^{[2]};~y\in \fg_1)$ is invariant under the action of $\fg$, the action of $\fg$ on $U(\fg)$ naturally extends to an action on $U_{\super}(\fg)$. The invariant subalgebra is denoted $U_{\super}(\fg)^{\fg}$. 
We write $Z_{\super}(\fg)$ for the center of $U_{\super}(\fg)$. 
In particular, we have $Z_{\super}(\fg)=U_{\super}(\fg)^{\fg}$.
There is one obvious family of central elements in $U_{\super}(\fg)$.
For any $r\in\NN$, we set 
\begin{align}\label{zr}
z_r:=e_{1,1}t^r+\cdots +e_{m+n,m+n}t^r.  
\end{align}
Using \eqref{lie bracket of g} one can show by direct computation that the set $\{z_r;~r\geq 0\}$ forms a basis for the center $\fz(\fg)$ of $\fg$.
Now Proposition \ref{prop:PBW for u superg} implies $\kk[z_r;~r\geq 0]$ is a subalgebra of $Z_{\super}(\fg)$.

There is a filtration 
\[
 U(\fg)=\bigcup\limits_{r\geq 0}{\rm F}_r U(\fg)
\]
of the enveloping algebra $U(\fg)$, 
which is defined by placing $e_{i,j}t^r$ in degree $r+1$.
The filtration descends to a filtration on $U_{\super}(\fg)$.
According to the PBW theorem (Proposition \ref{prop:PBW for u superg}) (see also \cite[Theorem 3.2]{EH25}), 
the associated graded algebra $\gr U_{\super}(\fg)$ is isomorphic to $S(\fg_{\bar 0})\otimes \Lambda (\fg_{\bar 1})$.
Here $\Lambda(\fg_{\bar 1})$ is the quotient of $S(\fg_{\bar 1})$ by the relations $x^2=0$, $x\in\fg_{\bar 1}$ (see the introduction of \cite{EH25}).
Let $S_{\super}(\fg):=S(\fg_{\bar 0})\otimes \Lambda (\fg_{\bar 1})$ for short. It follows that 
\begin{align}\label{grzsuper shuyu Ssuper gg}
\gr Z_{\super}(\fg)=\gr U_{\super}(\fg)^{\fg}\subseteq (\gr U_{\super}(\fg))^\fg=S_{\super}(\fg)^{\fg}.
\end{align}
\begin{Lemma}\label{Lem: invariant subalgebra S super gg}
The invariant algebra $S_{\super}(\fg)^\fg$ is generated by $\{z_r;~r\geq 0\}$ together with $(\fg_{\bar 0})^2:=\{x^2;~x\in\fg_{\bar 0}\}\subseteq S_{\super}(\fg)$.
In fact,
$S_{\super}(\fg)^\fg$ is freely generated by
\begin{align}\label{free generators of ssuper gg}
\{z_r;~r\geq 0\}\cup \{(e_{i,j}t^r)^2;~1\leq i, j\leq m+n~{\rm with}~(i,j)\neq (1,1), \bar i+\bar j=0\!\!\!\!\! \mod 2, r\geq 0\}.
\end{align}
\end{Lemma}
\begin{proof}
The proof is essentially the same as \cite[Lemma 2.1]{CH23} and \cite[Lemma 3.2]{BT18}. Let us point out the minor differences from therein.
Let $I(\fg)$ be the subalgebra of $S_{\super}(\fg)$ generated by $\{z_r;~r\geq 0\}$ and $(\fg_{\bar 0})^2$.
Let
\[
B_0:=\{(i,j,r);~1\leq i, j\leq m+n~{\rm with}~(i,j)\neq (1,1), \bar i+\bar j=0\!\!\!\!\! \mod 2, r\geq 0\},
\]
\[
B_1:=\{(i,j,r);~1\leq i, j\leq m+n~{\rm with}~\bar i+\bar j=1\!\!\!\!\! \mod 2, r\geq 0\},
\]
and $B:=B_0\cup B_1$ for short.
In \cite[Lemma 2.1]{CH23}, we don't have the odd elements if $p=2$.
The proof there follows immediately from \cite[Lemma 3.2]{BT18}.
However, in our situation, we consider the $S_{\super}(\fg)$ instead of $S(\fg)$.
We still have
\[
S_{\super}(\fg)=\kk[z_r;~r\geq 0][e_{i,j}t^r;~(i,j,r)\in B_0]\otimes \Lambda[e_{i,j}t^r;~(i,j,r)\in B_1],
\]
\[
I(\fg)=\kk[z_r;~r\geq 0][(e_{i,j}t^r)^2;~(i,j,r)\in B_0], 
\]
and $S_{\super}(\fg)$ is free as an $I(\fg)$-module.
Then we can repeat the proof of \cite[Lemma 2.1]{CH23}
without any change to obtain our desired result.
\end{proof}

We define the {\it $p$-center} of $U_{\super}(\fg)$ to be the image of the $p$-center of $U(\fg)$ in $U_{\super}(\fg)$,
and denote it by $Z_{p,\super}(\fg)$.
Recall from \eqref{zpg} that $Z_p(\fg)$ is free polynomial algebra generated by 
\[
\{(e_{i,j}t^r)^p-\delta_{i,j}e_{i,j}t^{rp};~1\leq i,j\leq m+n, r\geq 0\}.
\]
In particular, the image $Z_{p,\super}(\fg)$ is generated by 
\begin{align}\label{generators of zp super of g}
\{(e_{i,j}t^r)^p-\delta_{i,j}e_{i,j}t^{rp};~1\leq i,j\leq m+n, \bar i+\bar j=0\!\!\!\!\! \mod 2, r\geq 0\}.    
\end{align}

The following theorem is a super generalization of \cite[Theorem 3.4]{BT18} for $p=2$, see also \cite[Theorem 2.3]{CH23}.
The proof is similar to that of \cite[Theorem 3.4]{BT18}.
Since our setting is different, we provide a detailed proof here.
\begin{Theorem}\label{theorem: center Zsuper g}
The center $Z_{\super}(\fg)$ of $U_{\super}(\fg)$ is generated by $\{z_r;~r\geq 0\}$ and $Z_{p,\super}(\fg)$.
Moreover:
\begin{itemize}
\item[(1)] $Z_{p,\super}(\fg)$ is freely generated by the elements in \eqref{generators of zp super of g};
\item[(2)] $Z_{\super}(\fg)$ is freely generated by 
\begin{align}\label{generators of zsuper g}
\{z_r;~r\geq 0\}\cup\{(e_{i,j}t^r)^2-&\delta_{i,j}e_{i,j}t^{2r};~1\leq i, j\leq m+n~\\\nonumber
&{\rm with}~(i,j)\neq (1,1), \bar i+\bar j=0\!\!\!\!\! \mod 2, r\geq 0\}.      
\end{align}

\end{itemize}
\end{Theorem}
\begin{proof}
For $1\leq i,j\leq m+n$ with $\bar i+\bar j=0\!\!\!\mod 2$ and $r\geq 0$,
we have that the element $(e_{i,j}t^r)^2-\delta_{i,j}e_{i,j}t^{2r}$ is of filtered degree $2r+2$
and 
\[
\gr_{2r+2}[(e_{i,j}t^r)^2-\delta_{i,j}e_{i,j}t^{2r}]=(e_{i,j}t^r)^2\in S_{\super}(\fg).
\]
Since the elements $\{(e_{i,j}t^r)^2;~1\leq i,j\leq m+n, \bar i+\bar j=0 \!\!\! \mod 2, r\geq 0\}$ are algebraically independent in $S_{\super}(\fg)$,
it follows that the elements in \eqref{generators of zp super of g} are algebraically independent in $Z_{p,\super}(\fg)$, this proves (1).

Clearly, all the elements in \eqref{generators of zsuper g} are contained in $Z_{\super}(\fg)$.
Let $Z$ be the subalgebra of $Z_{\super}(\fg)$ generated by \eqref{generators of zsuper g}.
For $r\geq 0$ we have that $z_r\in{\rm F}_{r+1}U_{\super}(\fg)$ and $\gr_{r+1}z_r=z_r\in S_{\super}(\fg)$.
Lemma \ref{Lem: invariant subalgebra S super gg} implies that the elements \eqref{generators of zsuper g} are lifts of the algebraically independent generators of $S_{\super}(\fg)$ from \eqref{free generators of ssuper gg}.
It follows that the elements \eqref{generators of zsuper g} are themselves algebraically independent, and moreover $S_{\super}(\fg)^{\fg}\subseteq \gr Z$.
Thanks to \eqref{grzsuper shuyu Ssuper gg},
we also have
\[
\gr Z\subseteq \gr Z_{\super}(\fg)\subseteq S_{\super}(\fg)^{\fg},
\]
so equality must hold throughout.
This ensures that $Z=Z_{\super}(\fg)$.
\end{proof}
\section{Super Yangian in characteristic $2$}\label{section:super Yangian}
\subsection{Modular Yangian $Y_{m+n}$}
The Yangian $Y_{m+n}$ over $\kk$ is defined to be the associative algebra with generators $\{t_{i,j}^{(r)};~1\leq i,j\leq m+n, r>0\}$ subject just to the relations \eqref{RTT relations}.
Using {\em Gauss decomposition},
Brundan and Topley established the following modular analogue of \cite[Theorem 5.2]{BK05},
see \cite[Theorem 4.3]{BT18} and \cite[Theorem 3.8]{CH23}.
\begin{Theorem}\label{thm: Drinfled presentation ym+n}
The Yangian $Y_{m+n}$ is generated by the elements $\{d_i^{(r)},d_i'^{(r)};~1\leq i\leq m+n, r\geq 1\}$ and $\{e_j^{(r)},f_j^{(r)};~1\leq j\leq m+n-1, r\geq 1\}$ subject only to the following relations:
\begin{align}\label{di di' relation}
d_i^{(0)}=1,~\sum\limits_{t=0}^rd_i^{(t)}d_i'^{(r-t)}=\delta_{r0};
\end{align}
\begin{align}\label{di dj commu}
[d_i^{(r)},d_j^{(s)}]=0;
\end{align}
\begin{align}\label{Drinfeld generators relation 1}
[d_i^{(r)},e_j^{(s)}]=(\delta_{ij}-\delta_{i,j+1})\sum\limits_{t=0}^{r-1}d_i^{(t)}e_j^{(r+s-1-t)};
\end{align}
\begin{align}\label{Drinfeld generators relation 2}
[d_i^{(r)},f_j^{(s)}]=(\delta_{i,j+1}-\delta_{ij})\sum\limits_{t=0}^{r-1}f_j^{(r+s-1-t)}d_i^{(t)};
\end{align}
\begin{align}\label{Drinfeld generators relation 3}
[e_i^{(r)},f_j^{(s)}]=-\delta_{ij}\sum\limits_{t=0}^{r+s-1}d_i'^{(t)}d_{i+1}^{(r+s-1-t)};
\end{align}
\begin{align}\label{Drinfeld generators relation 4}
[e_j^{(r)},e_j^{(s)}]=(\sum\limits_{t=1}^{s-1}e_j^{(t)}e_j^{(r+s-1-t)}-\sum\limits_{t=1}^{r-1}e_j^{(t)}e_j^{(r+s-1-t)});
\end{align}
\begin{align}\label{Drinfeld generators relation 5}
[f_j^{(r)},f_j^{(s)}]=(\sum\limits_{t=1}^{r-1}f_j^{(t)}f_j^{(r+s-1-t)}-\sum\limits_{t=1}^{s-1}f_j^{(t)}f_j^{(r+s-1-t)});
\end{align}
\begin{align}\label{Drinfeld generators relation 6}
[e_j^{(r+1)},e_{j+1}^{(s)}]-[e_j^{(r)},e_{j+1}^{(s+1)}]=e_j^{(r)}e_{j+1}^{(s)};
\end{align}
\begin{align}\label{Drinfeld generators relation 7}
[f_j^{(r+1)},f_{j+1}^{(s)}]-[f_j^{(r)},f_{j+1}^{(s+1)}]=-f_{j+1}^{(s)}f_{j}^{(r)};
\end{align}
\begin{align}\label{Drinfeld generators relation 0}
[e_i^{(r)},e_j^{(s)}]=0=[f_i^{(r)},f_j^{(s)}]\quad\mbox{if}~~|i-j|>1;
\end{align}
\begin{align}\label{Drinfeld generators relation 8}
[[e_i^{(r)},e_j^{(s)}],e_j^{(t)}]+[[e_i^{(r)},e_j^{(t)}],e_j^{(s)}]=0,\quad\mbox{if}~~|i-j|=1;
\end{align}
\begin{align}\label{Drinfeld generators relation 9}
[[f_i^{(r)},f_j^{(s)}],f_j^{(t)}]+[[f_i^{(r)},f_j^{(t)}],f_j^{(s)}]=0,\quad\mbox{if}~~|i-j|=1;
\end{align}
\begin{align}\label{Drinfeld generators relation 10}
[[e_i^{(r)},e_j^{(t)}],e_j^{(t)}]=0\quad\mbox{if}~~|i-j|=1;
\end{align}
\begin{align}\label{Drinfeld generators relation 11}
[[f_i^{(r)},f_j^{(t)}],f_j^{(t)}]=0,\quad\mbox{if}~~|i-j|=1;
\end{align}
\begin{align}\label{Drinfeld generators relation 12}
[[e_{i-1}^{(r)},e_i^{(1)}],[e_i^{(1)},e_{i+1}^{(s)}]]=0;
\end{align}
\begin{align}\label{Drinfeld generators relation 13}
[[f_{i-1}^{(r)},f_i^{(1)}],[f_i^{(1)},f_{i+1}^{(s)}]]=0.
\end{align}
\end{Theorem}
The generators of $Y_{m+n}$ in Theorem \ref{thm: Drinfled presentation ym+n} above are called {\it Drinfeld generators}. 
In order to state the PBW theorem we define higher root elements as follows. 
For $i=1,\dots,m+n-1$, we set
\[
e_{i,i+1}^{(r)}:=e_i^{(r)},\quad  f_{i+1,i}^{(r)}:=f_{i}^{(r)}   
\]
and define inductively
\[
e_{i,j}^{(r)}:=[e_{i,j-1}^{(r)},e_{j-1}^{(1)}],\quad f_{j,i}^{(r)}:=[f_{j-1}^{(1)},f_{j-1,i}^{(r)}]
\]
for all $1\leq i<j\leq m+n$ and $r>0$.
In fact, the Drinfeld generators can be expressed in terms of {\it quasi-determinants} of \cite{GR97}, see \cite[Section 5]{BK05}.

The {\it loop filtration} on $Y_{m+n}$ is defined by placing the elements $e_{i,j}^{(r+1)}$,
$d_{i}^{(r+1)}$, $f_{j,i}^{(r+1)}$ in filtered degree $r$ for all $r\geq 0$. 
We write $\cF_r Y_{m+n}$ for the filtered piece of degree $r$,
so that $Y_{m+n}=\cup_{r\geq 0}\cF_r Y_{m+n}$,
and we write $\gr Y_{m+n}$ for the associated graded algebra.
By the proof of \cite[Theorem 4.3]{BT18} there is an isomorphism
\begin{align}\label{ug iso grymn}
\tilde{\psi}: U(\fg)\rightarrow \gr Y_{m+n}
\end{align}
defined by
\begin{equation}
\begin{aligned}\label{gr def}
e_{i,i}t^r\mapsto \gr_r d_i^{(r+1)}\\
e_{i,j}t^r\mapsto \gr_r e_{i,j}^{(r+1)}\\
e_{j,i}t^r\mapsto \gr_r f_{j,i}^{(r+1)}
\end{aligned}
\end{equation}
for $i<j$.
This gives the following PBW theorem for $Y_{m+n}$ in terms of Drinfeld generators \cite[Theorem 4.5]{BT18}.
\begin{Theorem}\label{thm:PBW ymn Drinfeld}
Ordered monomials in the elements
\begin{align}\label{Drinfeld generators}
 \{d_i^{(r)};~1\leq i\leq m+n, r>0\}\cup\{e_{i,j}^{(r)},f_{j,i}^{(r)};~1\leq i<j\leq m+n, r>0\}   
\end{align}
taken in any fixed ordering form a basis of $Y_{m+n}$.
\end{Theorem}
\subsection{Super Yangians in characteristic $2$}
Let $u$ be an indeterminate,
and consider the power series ring $Y_{m+n}[[u^{-1}]]$.
We define the power series
\[
d_i(u):=\sum\limits_{r\geq 0}d_i^{(r)}u^{-r}\in Y_{m+n}[[u^{-1}]]
\]
for all $i=1,\dots,m+n$.
We proceed to recall the description of the center of $Y_{m+n}$.
We first define
\begin{align}\label{cr}
c(u)=\sum\limits_{r\geq 0}c^{(r)}u^{-r}:=d_1(u)d_2(u-1)\cdots d_{m+n}(u-(m+n)+1).
\end{align}

According to \cite[Theorem 5.11(1)]{BT18} (see also \cite[Theorem 7.2]{BK05}),
the elements in $\{c^{(r)};~r>0\}$ are algebraically independent and lie in the center $Z(Y_{m+n})$ of $Y_{m+n}$.
The subalgebra they generate is called the {\it Harish-Chandra center} of $Y_{m+n}$, 
and is denoted $Z_{\HC}(Y_{m+n})$.

For $i=1,\dots,m+n$, we define
\begin{align}\label{biu}
b_i(u)=\sum_{r\geq 0}b_i^{(r)}:=d_i(u)d_i(u-1).
\end{align}
Remember that $p=2$.
By \cite[Theorem 5.11(2)]{BT18} the elements in
\begin{align}\label{generators of p-center of yangian}
\{b_i^{(2r)};~i=1,\dots,m+n,r>0\}\cup\{(e_{i,j}^{(r)})^2, (f_{j,i}^{(r)})^2;~1\leq i<j\leq m+n, r>0\}   
\end{align}
are algebraically independent, and lie in $Z(Y_{m+n})$.
The subalgebra they generate is called the {\it $p$-center} of $Y_{m+n}$ and is denoted $Z_p(Y_{m+n})$.
In fact, $Z_p(Y_{m+n})$ is a free polynomial algebra over the generators given in \eqref{generators of p-center of yangian}.
Moreover,
the $Z(Y_{m+n})$ is generated by $Z_{\HC}(Y_{m+n})$ and $Z_p(Y_{m+n})$.

In view of \eqref{RTT relations}, we can equip $Y_{m+n}=(Y_{m+n})_{\bar 0}\oplus (Y_{m+n})_{\bar 1}$ with the $\ZZ_2$-gradation, where the parity of $t_{i,j}^{(r)}$ is defined by $\bar i+\bar j\mod 2$. The elements of $(Y_{m+n})_{\bar 0}$ are called {\it even}, those of $(Y_{m+n})_{\bar 1}$ {\it odd}.
Using \cite[(4.6)-(4.9)]{BT18}, it is easy to see that all $d_i^{(r)}$ are even.
Moreover, the elements $e_{i,j}^{(r)}$ and $f_{j,i}^{(r)}$ are even (resp. odd) if $\bar i+\bar j=0\mod 2$ (resp. $\bar i+\bar j=1\mod 2$).

Let $Z_{p,\odd}$ be the subalgebra of $Z_p(Y_{m+n})$ generated by
\begin{align}\label{generators of p odd ym+n}
\{(e_{i,j}^{(r)})^2, (f_{j,i}^{(r)})^2;~1\leq i<j\leq m+n,~\bar i+\bar j=1\!\!\!\!\! \mod 2, r>0\}.
\end{align}
We define $Z_{p,\odd}(Y_{m+n})_+$ to be the maximal ideal of $Z_{p,\odd}(Y_{m+n})$ generated by the elements given in \eqref{generators of p odd ym+n}.
Now we can define the {\it super Yangian}
\begin{align}\label{definition of ym|n}
Y_{m|n}:=Y_{m+n}/Y_{m+n}Z_{p,\odd}(Y_{m+n})_+. 
\end{align}
\begin{Remark}
Actually, the following relations hold in $Y_{m|n}$ for all $r,s\geq 1$ (cf. \cite[(41)]{Gow07}):
\[
[e_{m}^{(r)},e_{m}^{(s)}]=0=[f_{m}^{(r)},f_{m}^{(s)}].   
\]
We can assume, without loss of generality, that $r<s$.
It follows from \eqref{Drinfeld generators relation 4} that
\[
[e_{m}^{(r)},e_{m}^{(s)}]=\sum\limits_{t=r}^{s-1}e_m^{(t)}e_m^{(r+s-1-t)}.    
\]
Since $(e_m^{(r)})^2=0$ in $Y_{m|n}$ for all $r>0$ by our definition,
the assertion immediately follows from the induction on $s-r$.
\end{Remark}
\begin{Theorem}\label{theorem: PBW theorem for ym|n}
The Yangian $Y_{m+n}$ is free as a module over $Z_{p,\odd}$ with basis given by the ordered supermonomials in the elements \eqref{Drinfeld generators}.
In particular, the images in $Y_{m|n}$ of the ordered supermonomials in the elements \eqref{Drinfeld generators} form a basis of $Y_{m|n}$.
\end{Theorem}
\begin{proof}
It suffices to show that the set consisting of ordered monomials in \eqref{generators of p odd ym+n} multiplied by ordered supermonomials in \eqref{Drinfeld generators} gives a basis for $Y_{m+n}$.
To see this,
we pass to the associated graded algebra \eqref{ug iso grymn} using \eqref{gr def} to reduce to showing that the monomials
\[
\prod_{\substack{1\leq i\neq j\leq m+n\\ \bar i+\bar j=1\!\!\!\!\! \mod 2\\r \geq 0}}
(e_{i,j}t^r)^{2a_{i,j;r}}
\prod_{\substack{1 \leq i\leq m+n\\r \geq 0}}
(e_{i,i}t^r)^{b_{i,i;r}}
\prod_{\substack{1\leq i\neq j\leq m+n\\ \bar i+\bar j=0\!\!\!\!\! \mod 2\\r \geq 0}}
(e_{i,j}t^r)^{b_{i,j;r}}
\prod_{\substack{1\leq i\neq j\leq m+n\\ \bar i+\bar j=1\!\!\!\!\! \mod 2\\r \geq 0}}
(e_{i,j}t^r)^{c_{i,j;r}}
\]
for $a_{i,j;r}\geq 0$, $b_{i,j;r}\geq 0$ and $0\leq c_{i,j;r}\leq 1$ form a basis for $U(\fg)$.
This is obvious.
\end{proof}

The loop filtration on $Y_{m+n}$ descends to a loop filtration on $Y_{m|n}$.
We denote the filtered pieces by $\cF_r Y_{m|n}$ for $ r\geq 0$,
and write $\gr Y_{m|n}$ for the associated graded algebra.
The next corollary follows from Theorem \ref{theorem: PBW theorem for ym|n}.
\begin{Corollary}\label{coro:u super g iso gr ym|n}
The isomorphism $\tilde{\psi}:U(\fg)\stackrel{\sim}{\longrightarrow}\gr Y_{m+n}$ given in \eqref{ug iso grymn} induces an isomorphism
\[
\psi:U_{\super}(\fg)\stackrel{\sim}{\longrightarrow}\gr Y_{m|n}.
\]
\end{Corollary}
\begin{proof}
The canonical projection $\pi:Y_{m+n}\rightarrow Y_{m|n}$ induces a surjective homomorphism $\gr\pi:\gr Y_{m+n}\rightarrow \gr Y_{m|n}$.
Given an odd element $e_{i,j}t^r$, 
the image of $(e_{i,j}t^r)^2)$ under $\tilde{\psi}$ is contained in the kernel of $\gr\pi$.
As a result,
$\tilde{\psi}$ induces a surjection $\psi:U_{\super}(\fg)\twoheadrightarrow\gr Y_{m|n}$.
We obtain the following commutative diagram.
\begin{center}
\begin{tikzcd}
U(\fg) \arrow[r, "\sim"',"\tilde{\psi}"] \arrow[d, two heads] & \gr Y_{m+n} \arrow[d, two heads,"\gr\pi"] \\
U_{\super}(\fg) \arrow[r, "\psi"']     & \gr Y_{m|n}
\end{tikzcd}     
\end{center}
Then we apply Proposition \ref{prop:PBW for u superg} in conjunction with Theorem \ref{theorem: PBW theorem for ym|n} to see that a basis for $U_{\super}(\fg)$ is sent to a basis for $\gr Y_{m|n}$.
Hence, $\psi$ is an isomorphism as required.
\end{proof}
\subsection{ The center of $Y_{m|n}$}
In this subsection,
we will describe the center $Z(Y_{m|n})$ of the super Yangian $Y_{m|n}$.
To do that we first recall that the Harish-Chandra center $Z_{\HC}(Y_{m+n})$ is generated by $\{c^{(r)};~r>0\}$ \eqref{cr}.
By abuse of notation, 
we still use the same notation $c^{(r)}$ to denote its image in $Y_{m|n}$.
In particular, the image of $Z_{\HC}(Y_{m+n})$ in $Y_{m|n}$ is generated by $\{c^{(r)};~r>0\}$.
We refer to this subalgebra of $Z(Y_{m|n})$ as the {\it Harish-Chandra center} of $Y_{m|n}$ and denote it by $Z_{\HC}(Y_{m|n})$.

Similarly we define the {\it $p$-center} of $Y_{m|n}$ to be the image of the $p$-center of $Y_{m+n}$ in $Y_{m|n}$, and denote it by $Z_p(Y_{m|n})$.
By \eqref{generators of p-center of yangian} and \eqref{definition of ym|n}, 
it is clear that $Z_p(Y_{m|n})$ is generated by 
\begin{equation}
\begin{aligned}\label{generators of zpym|n}
&\{b_i^{(2r)};~i=1,\dots,m+n,r>0\}\\
\cup\{(e_{i,j}^{(r)})^2,(f_{j,i}^{(r)})^2;&~1\leq i<j\leq m+n, \bar i+\bar j=0\!\!\!\!\! \mod 2, r>0\}.      
\end{aligned}
\end{equation}

The following theorem is a generalization of \cite[Theorem 5.11]{BT18}, see also \cite[Theorem 4.8]{CH23}.
\begin{Theorem}
The center $Z(Y_{m|n})$ is generated by $Z_{\HC}(Y_{m|n})$ and $Z_p(Y_{m|n})$.
Moreover:
\begin{itemize}
 \item[(1)] $Z_p(Y_{m|n})$ is the free polynomial algebra generated by the elements in \eqref{generators of zpym|n};
 \item[(2)] $Z(Y_{m|n})$ is the free polynomial generated by
 \begin{equation}
\begin{aligned}\label{free generators of zym|n}
&\{b_i^{(2r)}, c^{(r)};~2\leq i\leq m+n,r>0\}\\
\cup\{(e_{i,j}^{(r)})^2,(f_{j,i}^{(r)})^2;&~1\leq i<j\leq m+n, \bar i+\bar j=0\!\!\!\!\! \mod 2, r>0\}.      
\end{aligned}
\end{equation}
\end{itemize}
In particular, 
$Z({Y_{m|n}})$ coincides with the image of $Z(Y_{m+n})$ in $Y_{m|n}$.
\end{Theorem}
\begin{proof}
(1) According to \cite[Theorem 5.8]{BT18},   
we have that $b_i^{(2r)}\in\cF_{2r-2}Y_{m|n}$ and 
\begin{align}\label{gr b2r}
 \gr_{2r-2}b^{(2r)}=(e_{i,i}t^{r-1})^2-e_{i,i}t^{2r-2}\in U_{\super}(\fg).   
\end{align}
For $1\leq i<j\leq m+n$ with $\bar i+\bar j=0\!\!\!\mod 2$ and $r>0$,
we have by \eqref{gr def} that $(e_{i,j}^{(r)})^2$, $(f_{j,i}^{(r)})^2\in\cF_{2r-2}Y_{m|n}$ and 
\begin{align}\label{gr eijr2 and fjir2}
\gr_{2r-2}(e_{i,j}^{(r)})^2=(e_{i,j}t^{r-1})^2,\quad \gr_{2r-2}(f_{j,i}^{(r)})^2=(e_{j,i}t^{r-1})^2. 
\end{align}

By definition,
we just need to observe that the elements in \eqref{generators of zpym|n} are algebraically independent.
This follows because Theorem \ref{theorem: center Zsuper g} they are lifts of the algebraically independent generators of the $p$-center $Z_{p,\super}(\fg)$ of the associated graded algebra from \eqref{generators of zp super of g}.

(2) Let $Z$ be the subalgebra of $Z(Y_{m|n})$ generated by the given elements.
Then Corollary \ref{coro:u super g iso gr ym|n} yields
\begin{align}\label{baohan guanxi in zymn}
\gr Z\subseteq \gr Z(Y_{m|n})\subseteq Z(\gr Y_{m|n})=Z(U_{\super}(\fg))=Z_{\super}(\fg).
\end{align}
In view of \cite[Theorem 5.1]{BT18}, we have that $c^{(r)}\in\cF_{r-1}Y_{m|n}$ and 
\begin{align}\label{gr c}
\gr_{r-1}c^{(r)}=z_{r-1}.    
\end{align}
It follows from Theorem \ref{theorem: center Zsuper g}(2) that the generators of $Z$ are lifts of the algebraically independent generators of $Z_{\super}(\fg)$.
Hence, they are algebraically independent and equality holds everywhere in \eqref{baohan guanxi in zymn}. 
As a result, $Z=Z(Y_{m|n})$.
\end{proof}

\begin{Corollary}
The super Yangian $Y_{m|n}$ is free as module over its center, 
with basis given by the ordered monomials in
\begin{align}
\{d_i^{(r)};~2\leq i\leq m+n, r>0\}\}\cup\{e_{i,j}^{(r)},f_{j,i}^{(r)};~1\leq i<j\leq m+n, r>0\}
\end{align}
in which no exponent is $2$ or more.
\end{Corollary}
\begin{proof}
The proof, which uses \eqref{gr def}, \eqref{gr b2r}, \eqref{gr eijr2 and fjir2} and \eqref{gr c},
is similar to the proof of \cite[Corollary 5.12]{BT18}.
\end{proof}
Similarly, we have the following.
\begin{Corollary}
The super Yangian $Y_{m|n}$ is free as module over $Z_p(Y_{m|n})$, 
with basis given by the ordered monomials in
\begin{align}
\{d_i^{(r)};~1\leq i\leq m+n, r>0\}\}\cup\{e_{i,j}^{(r)},f_{j,i}^{(r)};~1\leq i<j\leq m+n, r>0\}
\end{align}
in which no exponent is $2$ or more.
\end{Corollary}
\begin{Remark}
We let $Z_p(Y_{m+n})_+$ be the maixmal ideal of $Z_p(Y_{m+n})$ generated by the elements given in \eqref{generators of p-center of yangian}.
The {\it restricted Yangian} is defined via 
\[
Y^{[p]}_{m+n}:=Y_{m+n}/Y_{m+n}Z_p(Y_{m+n})_+.
\]
This definition is due to Goodwin and Topley \cite[Section 4.3]{GT21}.
Similarly, 
we define $Z_p(Y_{m|n})_+$ to be the maximal ideal of $Z_p(Y_{m|n})$ generated by the elements \eqref{generators of zpym|n}.
We can define the {\it restricted super Yangian} 
\[
Y^{[p]}_{m|n}:=Y_{m|n}/Y_{m+n}Z_p(Y_{m|n})_+.   
\]
The definition \eqref{definition of ym|n} implies that $Y^{[p]}_{m+n}\cong Y^{[p]}_{m|n}$.
Recall that 
\[
U^{[p]}(\fg):=U(\fg)/(x^p-x^{[p]};~x\in \fg)
\]
is the {\it restricted enveloping algebra} of $\fg$ (cf. \cite[Section 2.7]{Jan98}),
while the restricted enveloping superalgebra is defined by
\[
U_{\super}^{[p]}(\fg):=U_{\super}(\fg)/(x^p-x^{[p]};~x\in \fg_{\bar 0}).   
\]
We also have $U^{[p]}(\fg)\cong U_{\super}^{[p]}(\fg)$.
This is not surprising because the quadratic map is the restriction of the $[p]$-map,
see also \cite{Bou+23} and \cite[Section 3.10]{EH25}.
\end{Remark}
\bigskip
\noindent
\textbf{Acknowledgment.}
This work is supported by the National Natural
Science Foundation of China (Grant No. 12461005) and
the Natural Science Foundation of Hubei Province (No. 2025AFB716).

\end{document}